\documentclass[11pt]{article}

\usepackage{amsmath,amssymb,amsthm}
\usepackage{geometry}
\usepackage{hyperref}

\newcommand{\supp}{\mathrm{supp}}

\theoremstyle{plain}
\newtheorem{theorem}{Theorem}[section]
\theoremstyle{definition}
\newtheorem{definition}[theorem]{Definition}
\theoremstyle{remark}
\newtheorem{remark}[theorem]{Remark}
\theoremstyle{definition}
\newtheorem{example}[theorem]{Example}
\DeclareMathOperator*{\esssup}{ess\,sup}

\begin{document}

\begin{center}
{\bf Estimates for tail functions under Riesz transforms in Grand Lebesgue Spaces}\\[3mm]
{ Maria Rosaria Formica$^{1}$, Eugene Ostrovsky$^{2}$ and Leonid Sirota$^{3}$.}\\[3mm]
\end{center}

\begin{center}
{\footnotesize $^{1}$ Universit\`a degli Studi di Napoli ``Parthenope'', via Generale Parisi 13,\\ 80132, Napoli, Italy.\\
e-mail: \texttt{mara.formica@uniparthenope.it}\\[1mm]
$^{2,3}$ Bar-Ilan University, Department of Mathematics and Statistics,\\
52900, Ramat Gan, Israel.\\
e-mail: \texttt{eugostrovsky@gmail.com}\\
e-mail: \texttt{sirota3@bezeqint.net}
}
\end{center}

\vspace{3mm}

%\begin{abstract}
%We study tail estimates for functions under generalized Riesz-type transformations in the framework of Grand Lebesgue Spaces. We derive general bounds from $L^p$ estimates, present explicit model examples, and apply the abstract result to the classical Riesz transforms.
%\end{abstract}

\begin{abstract}
We study the tail behaviour of measurable functions under generalized Riesz-type operators in the framework of Grand Lebesgue Spaces. By exploiting the connection between the growth of $L^p$ norms and the Young--Fenchel transform, we derive explicit tail estimates from suitable $L^p$ bounds. We also present model examples and apply the abstract result to the classical Riesz transforms, showing how the $L^p$ growth of the operator interacts with the intrinsic tail behaviour of the input function.
\end{abstract}

\vspace{3mm}

\noindent {\small {\sc Key words and phrases.} Riesz transformation, Lebesgue--Riesz norms and spaces, tail of a function, support of function, Chernoff's, Chebyshev--Markov inequalities, Grand Lebesgue Spaces and norms.}

\medskip

\noindent {\small {\sc Mathematics Subject Classification (2020).} Primary 42B20; Secondary 47G10, 46E30, 26D15.}

% 42B20 Singular and fractional integrals (Calderón-Zygmund, etc.)
% 46E30: Spaces of measurable functions ($L^p$-spaces, Orlicz spaces, Köthe function spaces, Lorentz spaces, etc.)
% 26D15: Inequalities for sums and integrals
% 47G10: Integral operators

\section{\bf Introduction and preliminaries}\label{Intro}

\vspace{3mm}

Let $d\in\mathbb N$. For $p\geq 1$ the classical Lebesgue-Riesz space $L^p(\mathbb{R}^d)$ consists of all measurable functions $f:\mathbb{R}^d \to \mathbb{R}$ having finite norm $\|f\|_{L^p(\mathbb{R}^d)}= \|f\|_p$, where
\begin{eqnarray*}
\|f\|_p  & := &\displaystyle
\left(\int_{\mathbb{R}^d}|f(x)|^p\,dx\right)^{1/p}, \ \ p\geq 1\\[2mm]
\|f\|_\infty & := & \displaystyle \esssup_{x\in\mathbb{R}^d}|f(x)|, \ \ p=\infty.
\end{eqnarray*}
For $1 \le a < b \le \infty$, we denote by $L(a,b)$ the intersection $\bigcap_{p\in(a,b)} L^p(\mathbb{R}^d)$.

\vspace{2mm}

Recall that for a measurable function $f:\mathbb{R}^d \to \mathbb{R}$, its {\it tail function} $T[f](t)$ is
\[
T[f](t) \stackrel{def}{=} \mu\{ x\in \mathbb{R}^d:\ |f(x)| \ge t \},\qquad t \ge 0,
\]
where $\mu$ is the Lebesgue measure. For general background on measure theory and integration, see e.g. \cite{Bogachev}, \cite{Gikhman Skorokhod}. A key tool in our analysis is Stein's identity:
\begin{equation}\label{Stein}
\|f\|_p^p = p \int_0^{\infty} t^{p-1} T[f](t)\,dt,\qquad p>0,
\end{equation}
see e.g. \cite{Stein}, Chapters 1--2.

\noindent
The Euclidean norm of $x\in\mathbb{R}^d$ is
\[
\|x\|=\sqrt{(x,x)}=\sqrt{\sum_{j=1}^d x_j^2},\qquad x=(x_1,\ldots,x_d).
\]

\begin{definition}[\sc Operators with Riesz-type $L^p$ growth]\label{def Riesz}
Let $1\le a_1<b_1\le\infty$. An operator $U$ (not necessarily linear)
acting on measurable functions on $\mathbb{R}^d$ is said to be of
\emph{Riesz type} if there exist constants $\alpha,\beta,C\ge0$ such that
for all $p\in(a_1,b_1)$ and $f\in L^p(\mathbb{R}^d)$ one has
\begin{equation}\label{Riesz:type}
\|U[f]\|_p
\le
C (p-a_1)^{-\alpha}(b_1-p)^{-\beta}\|f\|_p .
\end{equation}
\end{definition}

\noindent
This definition does not specify the operator explicitly; rather,
it describes a class of operators whose $L^p$ norms exhibit a
controlled growth near the endpoints $p=a_1$ and $p=b_1$.
Many classical singular integral operators, including the
Riesz transforms, satisfy estimates of this type.

%\begin{definition}[Generalized Riesz-type operator]\label{def Riesz}
%An operator $U$ (not necessarily linear) is said to be of {\it generalized Riesz type} if there exist parameters $1\le a_1<b_1\le\infty$ and constants $\alpha, \beta, C \ge 0$ such that for all $p\in(a_1,b_1)$ and all $f\in L^p(\mathbb{R}^d)$,
%\[
%\|U[f]\|_p \le C (p-a_1)^{-\alpha}(b_1-p)^{-\beta}\,\|f\|_p.
%\]
%\end{definition}
%The class of operators introduced in Definition 1.1 includes, in particular,
%Riesz-type operators and other singular integral operators whose $L^p$
%norms satisfy estimates of the form above.

\vspace{3mm}

\noindent
{\bf  Our goal is to establish upper tail estimates for $U[f]$ in terms of the tail behavior of the original function $f$.}

\vspace{3mm}

\noindent
For the general framework of functional analysis and operator theory, we refer to the classic monograph by Kantorovich and Akilov \cite{Kantorovich Akilov}. For examples and background  see also \cite{Iwaniec}, \cite{Pisier}, \cite{Stein}, \cite{Unser} etc.\par

\vspace{2mm}

\begin{definition}[\sc Grand Lebesgue Spaces]
Let $1\le a<b\le\infty$ and let $\psi:(a,b)\to(0,\infty]$ be a measurable function
such that $\inf_{p\in(a,b)}\psi(p)>0$.
The Grand Lebesgue Space $G\psi(a,b)$ consists of all measurable functions
$f:\mathbb{R}^d\to\mathbb{R}$ for which the norm
\[
\|f\|_{G\psi(a,b)}
:=
\sup_{p\in(a,b)}\frac{\|f\|_p}{\psi(p)}
\]
is finite.

\noindent The function $\psi$ is called the \emph{generating function} of the space.
The set of all such functions $\psi$ is denoted by $\Psi_{(a,b)}$ and
\[
\Psi \stackrel{def}{=} \bigcup_{(a,b):\,1 \le a < b \le \infty} \Psi_{(a,b)}.
\]
\end{definition}
\noindent  Denote $\supp(\psi) := \{ p:\ \psi(p) < \infty \}$.
For brevity, when the interval $(a,b)$ is fixed, we simply write
$G\psi$ and $\|f\|_{G\psi}$.

\vspace{2mm}

\begin{remark}[Iwaniec-Sbordone spaces]
For a domain $\Omega\subset\mathbb{R}^d$ of finite Lebesgue measure and $p \in (1, \infty)$,  $\theta\geq 0$, the classical space $L^{p),\theta}(\Omega)$ introduced by T. Iwaniec and C. Sbordone corresponds to $G\psi$ with $\psi(q) = (p-q)^{-\theta/q}$ for $q \in (1,p)$, choosing $q=p-\varepsilon$, \ $\varepsilon\in(0,p-1)$. In fact
\[
\sup_{q\in(1,p)}\frac{\|f\|_{L^q(\Omega)}}{\psi_{p,\theta}(q)}
=
\sup_{q\in(1,p)}(p-q)^{\theta/q}\|f\|_{L^q(\Omega)}
=
\sup_{0<\varepsilon <p-1}\varepsilon^{\frac{\theta}{p-\varepsilon}}\|f\|_{L^{p-\varepsilon}(\Omega)},
\]
which is the usual norm of the Lebesgue space $L^{p),\theta}(\Omega)$.
%Our framework generalizes this by allowing arbitrary generating functions and domains.

In particular, this shows that these spaces correspond to a specific choice of generating function
in the framework of the Grand Lebesgue spaces $G\psi$  on domains $\Omega$ of finite measure.
\end{remark}

\vspace{2mm}

\noindent
{\sc Natural function.}
For $g \in L(a,b)$, its {\it natural function} is defined by $\psi[g](p):=\|g\|_p$ for $p\in(a,b)$. By construction, $\|g\|_{G\psi[g]}=1$. Thus the function $\psi[g]$ describes the exact $L^p$-integrability
profile of $g$.
% This function captures the exact integrability profile of $g$.

\vspace{2mm}

Examples of generating functions $\psi \in \Psi_{(a,b)}$ frequently encountered in applications include:
\begin{align}
\psi_{a,b;\alpha,\beta}(p) &:= (p-a)^{-\alpha}(b-p)^{-\beta}, \quad p \in (a,b), \label{ex psi} \\
\psi_m(p) &:= p^{1/m}, \quad p \ge 1, \label{ex psi m}
\end{align}
where $1 \le a < b < \infty$, $\alpha, \beta \ge 0$, and $m > 0$.

The case $\psi_m(p)$ with $m=2$ corresponds to the well-known class of subgaussian
random variables, provided that the underlying measure $\mu$ is
probabilistic.

\vspace{3mm}

The theory of these spaces has been extensively developed and has many applications; see, for example,
\cite{FKOS}, \cite{Formica 14}, \cite{Kozachenko 1}, \cite{Kozachenko 11},
\cite{Kozachenko 12}, \cite{Liflyand}, \cite{Ostrovsky mono}, \cite{Ostr HAIT},
\cite{Yudovich1}, \cite{Yudovich2}.

\vspace{3mm}

\noindent
We now recall the classical connection between the tail function $T[f]$ and the Grand Lebesgue norm.

\noindent
By \eqref{Stein},
%\noindent\textit{Norm from tail.}
\begin{equation}\label{norm-from-tail}
\|f\|_{G\psi}
=\sup_{p\in\supp(\psi)}\dfrac{\left[p \, \displaystyle\int_0^{\infty} t^{p-1} T[f](t)\,dt\right]^{1/p}}{\psi(p)}.
\end{equation}

\vspace{3mm}

% =================================================================

Conversely, let $\|f\|_{G\psi}\in(0,\infty)$. By the Chebyshev-Markov inequality, for any $p \in \mathrm{supp}(\psi)$, we have
\begin{equation}\label{tail-from-norm-new}
T[f](t) \le \frac{\|f\|_p^p}{t^p} \le \frac{\|f\|_{G\psi}^p \cdot \psi^p(p)}{t^p}.
\end{equation}
By taking the logarithm, we can write
\[
\ln T[f](t) \le p \ln \psi(p) - p \ln t + p \ln \|f\|_{G\psi}.
\]
Define
\[
\nu_\psi(p):=p\ln\psi(p),\qquad p\in\supp(\psi),
\]
and let $\nu_\psi^*$ be its Young--Fenchel (Legendre) transform,
\[
\nu_\psi^*(y):=\sup_{p\in\supp(\psi)}\bigl(py-\nu_\psi(p)\bigr),
\qquad y\in\mathbb{R}.
\]
To obtain the sharpest estimate, we optimize over $p$. %Optimizing over $p$ yields
\begin{equation}\label{tail-from-norm}
T[f](t)\le \exp\left\{-\nu_\psi^*\!\left(\ln\frac{t}{\|f\|_{G\psi}}\right)\right\},
\qquad f\neq0,\ t>0.
\end{equation}

For brevity, define
\[
R[\psi](u;t):=\exp\left\{-\nu_\psi^*\!\left(\ln\frac{t}{u}\right)\right\},
\qquad u>0,\ t>0.
\]
Then
\[
T[f](t)\le R[\psi](\|f\|_{G\psi};t).
\]
The function $R[\psi](u;t)$ thus provides the corresponding tail estimate associated with the generating function~$\psi$.

%===========================================================================

\begin{remark}
The correspondence between the relations tail $\to$ norm \eqref{norm-from-tail} and norm $\to$ tail \eqref{tail-from-norm} is exact only in the exponential sense,
that is, up to multiplicative constants in the exponent; see, for instance,
\cite{FOS2024}, \cite{Liflyand}.
\end{remark}

\subsection*{Structure of the paper}

The paper is organized as follows.
In Section 2 we present explicit examples and compute the corresponding
$L^p$ norms and tail asymptotics.
In Section 3 we analyze how $L^p$ estimates for generalized Riesz-type
operators yield boundedness results in Grand Lebesgue Spaces.
Section 4 contains the main tail estimate for operators satisfying
suitable $L^p$ bounds.
Finally, in Section 5 we apply the abstract result to the classical
Riesz transforms.

\vspace{2mm}

%\section*{\sc Examples. Correct $L^p$ norms and tail regimes}

\section{\bf Examples: \texorpdfstring{$L^p$}{Lp} norms and tail asymptotics}

%In this section we present examples of functions for which the $L^p$ norms and
%the corresponding tail functions can be computed explicitly. Moreover, we determine the corresponding asymptotic form of the tail functions.  In the following examples, we compute the $L^p$ norms for several classes of radial functions; for the treatment of integrals involving functions with jump discontinuities, we refer to \cite{Grigorjeva}.

In this section we compute the $L^p$ norms and determine the asymptotic tail behavior for several classes of radial functions. These examples, involving functions with jump discontinuities (see e.g. \cite{Grigorjeva} for the general treatment of such integrals), allow for explicit calculations that illustrate the general theory.
We recall that $A(t)\sim B(t)$ as $t\to t_0$ means $\lim_{t\to t_0} A(t)/B(t)=1$.
%while $A(t)\asymp B(t)$ denotes comparability up to multiplicative constants.

\noindent
%We use polar coordinates.
We use radial coordinates in $\mathbb{R}^d$.
Let
\[
S_{d-1}=\frac{2\pi^{d/2}}{\Gamma(d/2)}
\]
denote the surface area of the unit sphere in $\mathbb{R}^d$, where $\Gamma(\cdot)$ is the Euler Gamma function
\[
\Gamma(z):=\int_0^\infty t^{z-1}e^{-t}\,dt,\qquad z>0.
\]
Then, for every radial integrable function $F(\|x\|)$,
\begin{equation}\label{radial function}
\int_{\mathbb{R}^d} F(\|x\|)\,dx
= S_{d-1}\int_0^\infty F(r)\,r^{d-1}\,dr.
\end{equation}
%We introduce the constants
%\begin{equation*}
%\kappa(d):=\int_0^\infty y^{d-1}e^{-y^2/2}\,dy
%=2^{\frac d2-1}\Gamma\!\left(\frac d2\right),
%\end{equation*}
%which follows from the substitution $u=y^2/2$, and
%\begin{equation*}
%J(d):=\frac{(2\pi)^{d/2}}{\kappa(d)}=\frac{2\pi^{d/2}}{\Gamma(d/2)}=S_{d-1}.
%\end{equation*}

\vspace{1mm}
We denote by $\mathbf{1}_A$ the indicator function of a set $A$, i.e.
\[
\mathbf{1}_A(x)=
\begin{cases}
1, & x\in A,\\
0, & x\notin A .
\end{cases}
\]

\begin{example}\label{ex:f}
Let
\[
f_{a,\gamma}(x)
:=
\mathbf{1}_{\{\|x\|>1\}}\,
\|x\|^{-1/a}(\ln \|x\|)^{\gamma},
\qquad a>0,\ \gamma\ge0.
\]
Let $r=\|x\|$. Then, using \eqref{radial function} and the substitution $u=\ln r$ (so that $r=e^u$, $dr=e^u\,du$), we obtain
\[
\|f_{a,\gamma}\|_p^p
=
S_{d-1}\int_1^\infty r^{d-1-p/a}(\ln r)^{\gamma p}\,dr
=
S_{d-1} \int_0^\infty e^{u(d-p/a)}u^{\gamma p}\,du .
\]
Applying  the Gamma identity
\[
\int_0^\infty e^{-\lambda u}u^\alpha\,du
=
\Gamma(\alpha+1)\lambda^{-\alpha-1}, \ \ \ \ \lambda>0, \ \ \alpha>-1
\]
with $\lambda=p/a-d$ and $\alpha=\gamma p$, for $p>ad$, we obtain
\[
\|f_{a,\gamma}\|_p^p =
S_{d-1}\,\Gamma(\gamma p+1)\,(p/a-d)^{-\gamma p-1},
\qquad p>ad.
\]

\vspace{3mm}

\noindent
{\bf Tail behaviour of $f_{a,\gamma}$.}
Consider the radial function
\[
\phi(r):=r^{-1/a}(\ln r)^\gamma,\qquad r> 1, \ \ \gamma>0.
\]
Since $\phi(r)\to 0$ as $r\to 1^+$ and $r\to\infty$, the continuous function $\phi$
is bounded on $(1,\infty)$ and therefore admits a finite supremum.

Since $|f_{a,\gamma}(x)|\le \|f_{a,\gamma}\|_\infty$ for all $x\in\mathbb{R}^d$,
the set $\{x:|f_{a,\gamma}(x)|\ge t\}$ is empty for $t>\|f_{a,\gamma}\|_\infty$,
and therefore
\[
T[f_{a,\gamma}](t)=0,\qquad t>\|f_{a,\gamma}\|_\infty .
\]
Observe that
\[
\frac{\phi'(r)}{\phi(r)}=\frac{d}{dr}\ln\phi(r).
\]
Since $\phi(r)>0$ for all $r>1$ and $\gamma>0$, we have
\[
\phi'(r)=0
\quad\Longleftrightarrow\quad
\frac{d}{dr}\ln\phi(r)=0.
\]
Therefore the critical points of $\phi$ coincide with those of $\ln\phi$.
If $\gamma>0$, then
\[
\frac{d}{dr}\ln \phi(r)
=
-\frac1{ar}+\frac{\gamma}{r\ln r},
\]
so that the unique critical point is given by
\[
\ln r=a\gamma,
\qquad\text{i.e.}\qquad
r=e^{a\gamma}.
\]
Therefore
\[
\|f_{a,\gamma}\|_\infty
=
\max_{r>1} r^{-1/a}(\ln r)^\gamma
=
e^{-\gamma}(a\gamma)^\gamma ,
\qquad \gamma>0.
\]
For $\gamma=0$ one has
\[
f_{a,0}(x)=\mathbf{1}_{\{\|x\|>1\}}\|x\|^{-1/a},
\]
that is,
\[
f_{a,0}(x)=0 \ \text{for} \ \|x\|\le1, \qquad
f_{a,0}(x)=\|x\|^{-1/a} \ \text{for} \ \|x\|>1.
\]
Since $r^{-1/a}$ is decreasing for $r>1$, its supremum on $(1,\infty)$
is attained as $r\to1^+$, and therefore
\[
\|f_{a,0}\|_\infty=1.
\]
The nontrivial tail behaviour corresponds to $t\downarrow0$.

For sufficiently small $t>0$, the equation
\[
\phi(r)=t  \ \quad \hbox{i.e.} \ \quad r^{-1/a}(\ln r)^\gamma=t
\]
has a unique large solution $r=r(t)>1$.

Since $\phi$ is decreasing for all sufficiently large $r$, we have
\[
\phi(r)\ge t
\quad\Longleftrightarrow\quad
1<r\le r(t).
\]
Since $f_{a,\gamma}(x)\ge t$ iff $1<\|x\|\le r(t)$, the level set
$\{x:f_{a,\gamma}(x)\ge t\}$ is the spherical shell
$1<\|x\|\le r(t)$. Then
\[
T[f_{a,\gamma}](t)=\mu\left(\{x\in\mathbb{R}^d:f_{a,\gamma}(x)\ge t\}\right)
=
\mu\left(\{x\in\mathbb{R}^d:1<\|x\|\le r(t)\}\right),
\]
and hence
\[
T[f_{a,\gamma}](t)
=
\mu\bigl(\|x\|\le r(t)\bigr)-\mu\bigl(\|x\|\le1\bigr)
=
\frac{S_{d-1}}{d}\bigl(r(t)^d-1\bigr).
\]
Since $r(t)\to\infty$ as $t\downarrow0$, this yields
\[
T[f_{a,\gamma}](t)\sim \frac{S_{d-1}}{d}\,r(t)^d,
\qquad t\downarrow0.
\]
Moreover, asymptotic inversion of
\[
t=r^{-1/a}(\ln r)^\gamma
\]
gives
\[
r(t)\sim t^{-a}\,[\ln(1/t)]^{a\gamma},
\qquad t\downarrow0.
\]
Consequently,
\begin{equation}\label{tail-f}
T[f_{a,\gamma}](t)
\sim
\frac{S_{d-1}}{d}\,
t^{-ad}\,[\ln(1/t)]^{a\gamma d},
\qquad t\downarrow0.
\end{equation}
Here $\sim$ means that the ratio of the two sides tends to $1$ as $t\downarrow0$.
\end{example}

\vspace{3mm}

\begin{example}\label{ex:g}
Let
\[
g_{b,\nu}(x)
:=
\mathbf{1}_{\{\|x\|<1\}}\,
\|x\|^{-1/b}\,|\ln \|x\||^{\nu},
\qquad b>0,\ \nu>0 .
\]
Setting $r=\|x\|$ and using \eqref{radial function}, we obtain
\[
\|g_{b,\nu}\|_p^p
=
S_{d-1}\int_0^1 r^{d-1-p/b}|\ln r|^{\nu p}\,dr .
\]

With the substitution $u=-\ln r$ (so that $r=e^{-u}$ and $dr=-e^{-u}du$), the integral becomes
\[
\|g_{b,\nu}\|_p^p
=
S_{d-1}\int_0^\infty e^{-u(d-p/b)}u^{\nu p}\,du .
\]

Applying the Gamma identity
\[
\int_0^\infty e^{-\lambda u}u^\alpha\,du
=
\Gamma(\alpha+1)\lambda^{-\alpha-1},
\qquad \lambda>0,\ \alpha>-1,
\]
with $\lambda=d-p/b$ and $\alpha=\nu p$, we obtain
\[
\|g_{b,\nu}\|_p^p
=
S_{d-1}\,\Gamma(\nu p+1)\,(d-p/b)^{-\nu p-1},
\qquad 1\le p<bd .
\]

\medskip
\noindent
{\bf Tail behaviour of $g_{b,\nu}$.}
Since $g_{b,\nu}(x)\to\infty$ as $\|x\|\downarrow0$, large values of
$g_{b,\nu}$ occur near the origin. Therefore the nontrivial tail
behaviour corresponds to the regime $t\to\infty$.

For sufficiently large $t>0$, let $r(t)\in(0,1)$ be the solution of
\[
r^{-1/b}|\ln r|^\nu=t .
\]

Since $g_{b,\nu}(x)\ge t$ iff $\|x\|\le r(t)$, the level set
$\{x:g_{b,\nu}(x)\ge t\}$ is the ball $\|x\|\le r(t)$. Hence
\[
T[g_{b,\nu}](t)
=
\mu(\|x\|\le r(t))
=
\frac{S_{d-1}}{d}\,r(t)^d .
\]

Moreover, asymptotic inversion of
\[
t=r^{-1/b}|\ln r|^\nu
\]
gives
\[
r(t)\sim t^{-b}(\ln t)^{b\nu},
\qquad t\to\infty .
\]

Consequently,
\begin{equation}\label{tail-g}
T[g_{b,\nu}](t)
\sim
\frac{S_{d-1}}{d}\,
t^{-bd}(\ln t)^{b\nu d},
\qquad t\to\infty .
\end{equation}

For $t\downarrow0$, one simply has
\[
T[g_{b,\nu}](t)\to \mu(\{\|x\|<1\})=\frac{S_{d-1}}{d}.
\]
\end{example}

\vspace{3mm}

\begin{remark}
The previous two examples are complementary. The function $f_{a,\gamma}$
describes the behaviour of tails corresponding to large radii
($r\to\infty$), which leads to the regime $t\downarrow0$.
Conversely, the function $g_{b,\nu}$ describes the behaviour near the
origin ($r\to0$), producing the opposite regime $t\to\infty$. Summarizing:
\end{remark}

%\begin{center}
%\begin{tabular}{|c|c|c|}
%\hline
%Property & Example $f_{a,\gamma}$ & Example $g_{b,\nu}$ \\
%\hline
%Definition &
%$\displaystyle \mathbf{1}_{\{\|x\|>1\}}\,
%\|x\|^{-1/a}(\ln \|x\|)^\gamma$
%&
%$\displaystyle \mathbf{1}_{\{\|x\|<1\}}\,
%\|x\|^{-1/b}|\ln \|x\||^\nu$
%\\
%\hline
%Radial domain &
%$r>1$ &
%$0<r<1$
%\\
%\hline
%Limit of the function &
%$f_{a,\gamma}(x)\to0$ as $r\to\infty$ &
%$g_{b,\nu}(x)\to\infty$ as $r\to0$
%\\
%\hline
%Interesting tail regime &
%$t\downarrow0$ &
%$t\to\infty$
%\\
%\hline
%Radial equation &
%$\displaystyle t=r^{-1/a}(\ln r)^\gamma$
%&
%$\displaystyle t=r^{-1/b}|\ln r|^\nu$
%\\
%\hline
%Asymptotic inversion &
%$\displaystyle r(t)\sim t^{-a}[\ln(1/t)]^{a\gamma}$
%&
%$\displaystyle r(t)\sim t^{-b}(\ln t)^{b\nu}$
%\\
%\hline
%Tail behaviour &
%$\displaystyle T[f](t)\sim
%\frac{S_{d-1}}{d}\,
%t^{-ad}[\ln(1/t)]^{a\gamma d}$
%&
%$\displaystyle T[g](t)\sim
%\frac{S_{d-1}}{d}\,
%t^{-bd}(\ln t)^{b\nu d}$
%\\
%\hline
%\end{tabular}
%\end{center}

\begin{center}
\begin{tabular}{|c|c|c|}
\hline
Function & Behaviour & Tail regime \\
\hline
$f_{a,\gamma}=\displaystyle \mathbf{1}_{\{\|x\|>1\}}\,
\|x\|^{-1/a}(\ln \|x\|)^\gamma$
&
$\displaystyle f_{a,\gamma}(x)\to 0 \quad (\|x\|\to\infty)$
&
$\displaystyle t\downarrow 0$
\\
\hline
$g_{b,\nu}=\displaystyle \mathbf{1}_{\{\|x\|<1\}}\,
\|x\|^{-1/b}|\ln \|x\||^\nu$
&
$\displaystyle g_{b,\nu}(x)\to\infty \quad (\|x\|\to 0)$
&
$\displaystyle t\to\infty$
\\
\hline
\end{tabular}
\end{center}

\vspace{3mm}

\begin{example}\label{ex:sum}
Let
\[
h_{a,b;\gamma,\nu}(x)
:=
f_{a,\gamma}(x)+g_{b,\nu}(x),
\]
where $f_{a,\gamma}$ and $g_{b,\nu}$ are as in the previous examples.

Since $f_{a,\gamma}$ is supported on $\{\|x\|>1\}$ and
$g_{b,\nu}$ is supported on $\{\|x\|<1\}$, their supports are disjoint
(up to the null set $\{\|x\|=1\}$). We have
$f_{a,\gamma}(x)g_{b,\nu}(x)=0$ almost everywhere. Hence
\[
|f_{a,\gamma}(x)+g_{b,\nu}(x)|^p
=
|f_{a,\gamma}(x)|^p+|g_{b,\nu}(x)|^p
\quad \text{a.e.}
\]
and therefore, for any $p\in(ad,bd)$,
\[
\|h_{a,b;\gamma,\nu}\|_p^p
=
\|f_{a,\gamma}\|_p^p
+
\|g_{b,\nu}\|_p^p .
\]

\medskip
\noindent
{\bf Tail behaviour of $h_{a,b;\gamma,\nu}$.}
For small $t$, the outer component $f_{a,\gamma}$ dominates, whereas
for large $t$ the singular behaviour of $g_{b,\nu}$ near the origin
dominates. Consequently,
\begin{equation}\label{tail-h-small}
T[h_{a,b;\gamma,\nu}](t)
\sim
T[f_{a,\gamma}](t)
\sim
\frac{S_{d-1}}{d}\,
t^{-ad}[\ln(1/t)]^{a\gamma d},
\qquad t\downarrow0,
\end{equation}
and
\begin{equation}\label{tail-h-large}
T[h_{a,b;\gamma,\nu}](t)
\sim
T[g_{b,\nu}](t)
\sim
\frac{S_{d-1}}{d}\,
t^{-bd}(\ln t)^{b\nu d},
\qquad t\to\infty .
\end{equation}
\end{example}
%The function $h_{a,b;\gamma,\nu}](t)$ combines both behaviours of $f_{a,\gamma}$ and $g_{b,\nu}$.

\begin{remark}
The above examples illustrate two complementary types of tail behaviour.
The function $f_{a,\gamma}$ produces a tail regime corresponding to
$t\downarrow0$, while $g_{b,\nu}$ produces the opposite regime $t\to\infty$.
The function $h_{a,b;\gamma,\nu}=f_{a,\gamma}+g_{b,\nu}$ combines both
effects.

In particular, the corresponding natural functions
\[
\psi[f_{a,\gamma}](p)=\|f_{a,\gamma}\|_p,
\qquad
\psi[g_{b,\nu}](p)=\|g_{b,\nu}\|_p
\]
describe explicitly the growth of the $L^p$ norms and therefore generate
the associated Grand Lebesgue Spaces.
\end{remark}

\vspace{5mm}

%\section{Auxiliary facts: GLS bounds derived from \texorpdfstring{$L^p$}{Lp} estimates}

\section{Boundedness of Generalized Riesz-type Operators in \texorpdfstring{$G\psi$}{G psi} spaces}

Let $U$ be a generalized Riesz-type operator as introduced in Definition \ref{def Riesz}. In this section, we show how its $L^p$ properties translate into boundedness within the framework of Grand Lebesgue Spaces. This step is essential because, although $U$ is defined via point-wise $L^p$ estimates, its action on the tail behavior is best captured by the norm of a Grand Lebesgue space $G\varphi$, whose generating function $\varphi$ will be constructed starting from the operator's profile.

\medskip
\noindent
Let $1\le a_1<b_1\le\infty$ and $1\le a_2<b_2\le\infty$, and let
\[
h:\mathbb{R}^d\to\mathbb{R}
\]
be a measurable function. Assume that
\[
\zeta\in\Psi_{(a_1,b_1)}, \qquad \psi\in\Psi_{(a_2,b_2)}.
\]
Suppose that an operator (not necessarily linear) $U$ satisfies the estimate
\begin{equation}\label{Lp-bound}
\|U[h]\|_p \le \zeta(p)\,\|h\|_p,
\qquad p\in(a_1,b_1),
\end{equation}
and that
\[
h\in G\psi(a_2,b_2).
\]
Let
\[
(a_3,b_3):=(a_1,b_1)\cap(a_2,b_2),
\]
and assume that this interval is nonempty. Define
\[
\varphi(p):=\zeta(p)\psi(p),\qquad p\in(a_3,b_3).
\]
Since $h\in G\psi(a_2,b_2)$, we have
\[
\|h\|_p\le \psi(p)\|h\|_{G\psi(a_2,b_2)},
\qquad p\in(a_2,b_2).
\]
Combining this with \eqref{Lp-bound}, we obtain for $p\in(a_3,b_3)$
\[
\|U[h]\|_p
\le
\zeta(p)\|h\|_p
\le
\zeta(p)\psi(p)\|h\|_{G\psi(a_2,b_2)}
=
\varphi(p)\|h\|_{G\psi(a_2,b_2)}.
\]
Taking the supremum over $p\in(a_3,b_3)$ yields
\begin{equation}\label{General}
\|U[h]\|_{G\varphi(a_3,b_3)}
\le
\|h\|_{G\psi(a_2,b_2)}.
\end{equation}
This general embedding between Grand Lebesgue spaces will be specialized in the next section to derive tail estimates by choosing the generating function $\psi$ as the natural function of the source $f$.

The inequality \eqref{General} provides a fundamental estimate for the operator $U$. We emphasize that this result represents an extension of the standard definition of the generalized Riesz operator. While the classical operator theory is typically developed for fixed $L^p$ spaces, the bound in \eqref{General} establishes a broader continuity across a family of Grand Lebesgue Spaces with varying parameters. This mapping property effectively characterizes the operator $U$ as a bounded transformation between the scales of spaces $G\psi(a_2,b_2)$ and $G\varphi(a_3,b_3)$, thereby extending its domain and range to a more general functional setting.

\subsection*{Optimization of the assumed bounds}

Given the $L^p$ estimate assumed in \eqref{Lp-bound}, it is fundamental to determine the quantitative behavior of the constant with respect to the parameters. To this end, we consider
\begin{equation}\label{ZetaDef}
\zeta(p) = C(p-a_1)^{-\alpha}(b_1-p)^{-\beta}, \quad p \in (a_1, b_1),
\end{equation}
which has been introduced in Definition \ref{def Riesz}.

The minimization of $\zeta(p)$ is a necessary step to ensure the sharpness of the embedding, following the classical variational approach for operator constants discussed by Kantorovich and Akilov \cite{Kantorovich Akilov}.

By analyzing the logarithmic derivative of \eqref{ZetaDef}, we identify the stationary point $p^* = \frac{\alpha b_1 + \beta a_1}{\alpha + \beta}$. Consequently, the infimum is given by:
\begin{equation}
\inf_{p \in (a_1, b_1)} \zeta(p) = \zeta(p^*) = \frac{C}{(b_1 - a_1)^{\alpha+\beta}} \cdot \frac{(\alpha+\beta)^{\alpha+\beta}}{\alpha^\alpha \beta^\beta}.
\end{equation}
This optimal value characterizes the growth of the operator norm as $p$ approaches the boundaries of the interval $(a_1, b_1)$, providing the sharpest possible weight for the Grand Lebesgue norm construction that follows.

\vspace{3mm}

\section{Tail estimates for operators with \texorpdfstring{$L^p$}{Lp} bounds}

Let $f:\mathbb{R}^d\to\mathbb{R}$ be measurable and assume that
\[
\|f\|_p<\infty,
\qquad p\in(a,b)
\]
for some interval $(a,b)\subset[1,\infty)$. Define the natural function
\[
w(p):=\|f\|_p,
\qquad p\in(a,b).
\]
Then by definition
\[
\|f\|_{G w(a,b)}=1 .
\]

%\medskip
%
%\noindent
%Let $U$ be an operator satisfying the estimate
%\begin{equation}\label{Lp-bound}
%\|U[h]\|_p \le \zeta(p)\|h\|_p,
%\qquad p\in(a_1,b_1),
%\end{equation}
%for some function $\zeta\in\Psi_{(a_1,b_1)}$.

The theorem below can be viewed as a direct consequence of the transfer
estimate \eqref{General} with the generating function $\psi=w$.

We recall that, for $1 \le a < b < \infty$,  $L(a,b)$ denote the intersection
$$L(a,b)=\bigcap_{p\in(a,b)} L^p(\mathbb{R}^d).$$

%========================================================================

\begin{theorem}
\label{thm:main}
Let $f\in L(a,b)$, $1\leq a<b<\infty$. Let $1\leq a_1<b_1<\infty$ and
$\zeta\in\Psi_{(a_1,b_1)}$. Assume
\[
I = (a,b)\cap (a_1,b_1)\neq \emptyset.
\]
Let $U$ be a generalized Riesz-type operator satisfying
\[
\|U[f]\|_p \le \zeta(p)\|f\|_p,
\qquad p\in I.
\]
Define
\[
w(p):=\|f\|_p,\qquad p\in(a,b),
\]
\[
\varphi_f(p):=w(p)\zeta(p),\qquad p\in I,
\]
and
\[
\nu_f(p):=p\ln\varphi_f(p),\qquad p\in I.
\]
Then
\[
\|U[f]\|_{G\varphi_f(I)}\le1.
\]
Consequently, for every $t>0$,
\begin{equation}\label{main-tail}
T[U(f)](t)
\le
\exp\{-\nu_f^*(\ln t)\},
\end{equation}
where
\[
\nu_f^*(y)
:=
\sup_{p\in I}\bigl(py-\nu_f(p)\bigr),
\qquad y\in\mathbb{R}.
\]
Equivalently, in the notation of \eqref{tail-from-norm},
\[
T[U(f)](t)\le R[\varphi_f](1;t).
\]
\end{theorem}

\begin{proof}
Let $p\in I$. Since $p\in(a_1,b_1)$, the estimate \eqref{Lp-bound}
applies to $f$, giving
\[
\|U[f]\|_p\le \zeta(p)\|f\|_p
= \zeta(p)w(p)
= \varphi_f(p).
\]
Hence
\[
\frac{\|U[f]\|_p}{\varphi_f(p)}\le1 .
\]
Taking the supremum over $p\in I$ yields
\[
\|U[f]\|_{G\varphi_f(I)}\le1 .
\]

We can now apply the tail estimate \eqref{tail-from-norm} with
$\psi=\varphi_f$ and $U[f]$ in place of $f$. Hence, for every $t>0$,
\[
T[U(f)](t)
\le
\exp\left\{-\nu_f^*\!\left(
\ln\frac{t}{\|U[f]\|_{G\varphi_f(I)}}
\right)\right\}.
\]
Since $\|U[f]\|_{G\varphi_f(I)}\le 1$, we have
\[
\ln\frac{t}{\|U[f]\|_{G\varphi_f(I)}}\ge \ln t.
\]
Moreover, the Young--Fenchel transform $\nu_f^*$ is nondecreasing on
$\mathbb{R}$, therefore
\[
\nu_f^*\!\left(
\ln\frac{t}{\|U[f]\|_{G\varphi_f(I)}}
\right)\ge \nu_f^*(\ln t).
\]
Consequently,
\[
T[U(f)](t)\le \exp\{-\nu_f^*(\ln t)\},
\qquad t>0.
\]
Equivalently, in the notation introduced in section \ref{Intro},
\[
R[\varphi_f](u;t):=\exp\left\{-\nu_f^*\!\left(\ln\frac{t}{u}\right)\right\},
\qquad u>0,\ t>0,
\]
we have
\[
T[U(f)](t)\le R[\varphi_f](1;t).
\]
\end{proof}

%==========================================================================

\begin{remark}
In particular, theorem \ref{thm:main} applies to generalized Riesz-type operators
in the sense of Definition \ref{def Riesz}, with
\[
\zeta(p)=C(p-a_1)^{-\alpha}(b_1-p)^{-\beta}.
\]
\end{remark}

\vspace{3mm}

%===============================================================

\noindent
{\sc Example (application to the model tails).}
We recall that the notation $A(t)\asymp B(t)$ means that $A(t)$ and $B(t)$ are comparable up to multiplicative constants.

Let $f$ be a measurable function such that
\[
T[f](t)\asymp T[f_{a,\gamma}](t),
\qquad t>0,
\]
where $f_{a,\gamma}$ is the model function introduced
in Example~\ref{ex:f}.
Then, by the moment--tail relation \eqref{Stein}, we have
\[
\|f\|_p^p = p\int_0^\infty t^{p-1}T[f](t)\,dt
\asymp
p\int_0^\infty t^{p-1}T[f_{a,\gamma}](t)\,dt
= \|f_{a,\gamma}\|_p^p .
\]
Consequently,
\[
\|f\|_p \asymp \|f_{a,\gamma}\|_p,
\qquad p>ad.
\]

An analogous argument applies when
\[
T[f](t)\asymp T[g_{b,\nu}](t),
\]
where $g_{b,\nu}$ is the model function introduced in Example~\ref{ex:g},
yielding
\[
\|f\|_p \asymp \|g_{b,\nu}\|_p,
\qquad p<bd.
\]
Hence, since
\[
\varphi_f(p)=\zeta(p)w(p)=\zeta(p)\|f\|_p,
\]
we have
\[
\varphi_f(p)\asymp \zeta(p)\|f_{a,\gamma}\|_p
\quad \text{or} \quad
\varphi_f(p)\asymp \zeta(p)\|g_{b,\nu}\|_p .
\]
Define
\[
\nu_f(p):=p\ln\varphi_f(p).
\]
Substituting this into the estimate of
Theorem~\ref{thm:main}, we obtain an explicit exponential bound
for the tail of $U[f]$:
\[
T[U(f)](t)\le \exp\{-\nu_f^*(\ln t)\}.
\]
Thus, the main theorem provides explicit exponential upper bounds
for the tails of $U[f]$ in terms of the model behaviours
derived in the previous examples.
%===========================================================

%===========================================================

\section{Application: the Riesz transform}

For $j=1,\dots,d$, let $R_j$ denote the $j$-th Riesz transform, defined by
\[
R_j f(x)
=
c_d\,\mathrm{p.v.}\int_{\mathbb{R}^d}
\frac{x_j-y_j}{|x-y|^{d+1}}\,f(y)\,dy,
\qquad x\in\mathbb{R}^d.
\]

It is well known that the Riesz transforms are bounded operators on
$L^p(\mathbb{R}^d)$ for $1<p<\infty$ (see, e.g., \cite[Theorem 12.1.1]{IwaniecMartinBook}).
For recent results on the $L^p$-boundedness of related operators on lattices, see \cite{Huang-Yao}.
In particular, their operator norm satisfies an estimate of the form
\begin{equation}\label{riesz-lp}
\|R_j f\|_p \le \zeta(p)\,\|f\|_p,
\qquad 1<p<\infty,
\end{equation}
where one may take
\[
\zeta(p)=C_d\,\max\!\left\{p,\frac{p}{p-1}\right\}.
\]
Hence the Riesz transforms fit the \texorpdfstring{$L^p$}{Lp}-bound framework
of Theorem~\ref{thm:main} with $\zeta(p)$ as defined above.

\noindent
The endpoint growth of $\zeta(p)$ is analogous to that appearing in
Definition~\ref{def Riesz}. More precisely,
\[
\zeta(p)\asymp \frac{1}{p-1}\qquad \text{as } p\downarrow1,
\qquad
\zeta(p)\asymp p\qquad \text{as } p\to\infty.
\]
Assume in addition that
\begin{equation}\label{hyp:ad}
ad>1.
\end{equation}
Suppose now that
\[
T[f](t)\asymp T[f_{a,\gamma}](t),
\qquad t\downarrow 0.
\]
Then, by the moment--tail relation \eqref{Stein} and Example~\ref{ex:f},
\[
w(p)=\|f\|_p
\asymp
\|f_{a,\gamma}\|_p
\asymp
\left[
S_{d-1}\Gamma(\gamma p+1)(p/a-d)^{-\gamma p-1}
\right]^{1/p},
\qquad p\downarrow ad.
\]

The moment estimate for the model function $f_{a,\gamma}$ is valid for
$p>ad$, while the bound \eqref{riesz-lp} for the Riesz transform holds
for $p>1$. Under the assumption $ad>1$, both estimates are simultaneously valid for
\[
p>ad.
\]
On this domain, the function 
\[
\varphi_f(p)=\zeta(p)w(p)
\]
is well-defined and, since $\zeta(p)$ remains bounded as $p\downarrow ad$, its asymptotic behaviour is driven exclusively by $w(p)$:
\[
\varphi_f(p)
\asymp
\max\!\left\{p,\frac{p}{p-1}\right\}
\left[
\frac{S_{d-1}\Gamma(\gamma p+1)}
{(p/a-d)^{\gamma p+1}}
\right]^{1/p}.
\]
Define
\[
\nu_f(p):=p\ln\varphi_f(p).
\]
Applying Theorem~\ref{thm:main}, we obtain
\[
T[R_j f](t)\le \exp\{-\nu_f^*(\ln t)\}.
\]
This shows explicitly how the $L^p$ growth of the Riesz transform combines
with the intrinsic singular behaviour of the model tail of $f$ to determine
the final tail estimate for $R_j f$.

%====================================================

\vspace{3mm}

\section{Conclusion}

%=====================================================

We have shown that the natural function $w(p)=\|f\|_p$ provides an effective
description of the tail behaviour of $f$. Once the action of an operator $U$
on $L^p$ norms is controlled through a function $\zeta(p)$, one can derive
tail estimates for $U[f]$ by studying the Young--Fenchel transform of
\[
p\mapsto p\ln\bigl(\zeta(p)w(p)\bigr).
\]
This method is particularly useful for operators that are not bounded on
$L^1$ or $L^\infty$, as frequently occurs in harmonic analysis.

%=====================================================

\vspace{3mm}

\section*{Acknowledgements}
M. R. Formica is a member of Gruppo Nazionale per l'Analisi Matematica, la Probabilit\`{a} e le loro Applicazioni (GNAMPA) of the Istituto Nazionale di Alta Matematica (INdAM) and a member of the UMI group ``Teoria dell'Approssimazione e Applicazioni (T.A.A.)''.


\begin{thebibliography}{99}

\bibitem{Bogachev}
{\bf V. I. Bogachev}, \emph{Measure Theory}, Vol. 1 and 2. Springer-Verlag, Berlin, Heidelberg, 2007.
ISBN 978-3-540-34513-8.

%\bibitem{Buldygin}
%{\bf V. V. Buldygin, D. I. Mushtary, E. I. Ostrovsky and M. I. Pushalsky.}
%\emph{New Trends in Probability Theory and Statistics}. Mokslas (1992), V.1, 78--92.

\bibitem{FKOS}
{\bf M. R. Formica, Y. V. Kozachenko, E. Ostrovsky and L. Sirota}, Exponential tail estimates in the law of ordinary logarithm (LOL) for triangular arrays of random variables. \emph{Lith. Math. J.}, \textbf{60} no.3, (2020), 330--358.

\bibitem{FOS2024}
{\bf M.R. Formica, E. Ostrovsky and L. Sirota},
Connection between weighted tail, Orlicz, Grand Lorentz and Grand Lebesgue norms.
\emph{Results Math.} {\bf 79}, 103 (2024).
\url{https://doi.org/10.1007/s00025-024-02136-0}


\bibitem{Formica 14}
{\bf M.R. Formica, E. Ostrovsky and L. Sirota.}
Moments and tail reciprocal connections for the random variables
having generalized Gamma--Weibull distributions.
\emph{arXiv:2206.00624 [math.PR]}.


\bibitem{Gikhman Skorokhod}
{\bf I. I. Gikhman and A. V. Skorokhod},
\emph{The Theory of Stochastic Processes I}.
Springer-Verlag, Berlin, Heidelberg, 2004. [Reprint of the 1974 English Edition;
originally published in Russian, Nauka, Moscow, 1971].

\bibitem{Grigorjeva}
{\bf M.L. Grigorjeva and E.I. Ostrovsky},
Calculation of Integrals on discontinuous functions by means of depending trials method.
\emph{Comput. Math. Math. Phys.} {\bf 36} (1996), no. 12, 1639--1648.

\bibitem{Huang-Yao}
{\bf S. Huang and X. Yao},
The $\ell^p$-boundedness of wave operators for the fourth order Schr\"odinger operators on the lattice $\mathbb{Z}$.
\emph{arXiv:2512.10649 [math.AP]} (11 Dec. 2025).

\bibitem{Iwaniec}
{\bf T. Iwaniec and G. Martin},
Riesz transforms and related Singular Integrals.
\emph{J. reine angew. Math.}, \textbf{473} (1996), 25--57.

\bibitem{IwaniecMartinBook}
{\bf T. Iwaniec and G. Martin},
\emph{Geometric Function Theory and Non-linear Analysis}.
Oxford Mathematical Monographs, Oxford University Press, 2001.

\bibitem{Kantorovich Akilov}
{\bf L. V. Kantorovich and G. P. Akilov},
\emph{Functional Analysis}.
3rd ed., Nauka, Moscow, 1984 (in Russian).
[English translation of the 2nd ed.: Pergamon Press, Oxford, 1982].


\bibitem{Kozachenko 1}
{\bf Yu.V. Kozachenko and E.I. Ostrovsky},
Banach spaces of random variables of sub-Gaussian type.
\emph{Theory Probab. Math. Statist.} {\bf 32} (1985), 42--53.

\bibitem{Kozachenko 11}
{\bf Yu.V. Kozachenko, E.I. Ostrovsky and L. Sirota},
Relations between exponential tails, moments and moment generating functions for random variables and vectors. \\
\emph{arXiv:1701.01901}.

\bibitem{Kozachenko 12}
{\bf Yu.V. Kozachenko, E.I. Ostrovsky and L. Sirota},
Equivalence between tails, Grand Lebesgue Spaces and Orlicz norms for random variables without Kramer's condition. \emph{Bulletin of the Taras Shevchenko National University of Kyiv} (2018), 4, 20--29. {\bf 4} (2018), 20--29.


\bibitem{Liflyand}
{\bf E. Liflyand, E. Ostrovsky and L. Sirota},
Structural properties of bilateral Grand Lebesgue Spaces.
\emph{Turkish J. Math.} {\bf 34} (2010), no. 2, 207--219.


\bibitem{Ostrovsky mono}
{\bf E.I. Ostrovsky},
\emph{Exponential Estimates for Random Fields} (in Russian).
Moscow--Obninsk, OINPE, 1999.

\bibitem{Ostr HAIT}
{\bf E. Ostrovsky and L. Sirota},
Moment Banach spaces: Theory and applications.
\emph{HAIT J. Sci. Eng. C} {\bf 4} (2007), no. 1-2, 233--262.
%HAIT Journal of Science and Engineering C


\bibitem{Pisier}
{\bf G. Pisier},
Riesz transforms: a simpler analytic proof of P.A. Meyer's inequality.
\emph{Lecture Notes in Math.} {\bf 1321} (1988), 485--501.

\bibitem{Stein}
{\bf E. M. Stein},
\emph{Topics in Harmonic Analysis Related to the Littlewood--Paley Theory}.
Ann. of Math. Stud., 63, Princeton Univ. Press, Princeton, NJ, 1970.


\bibitem{Unser}
{\bf M. Unser and D. Van De Ville},
Higher-order Riesz transforms and steerable wavelet frames.
In: \emph{Proc. IEEE Int. Conf. Image Process. (ICIP)},
2009, pp.~3801--3804.


\bibitem{Yudovich1}
{\bf V. I. Yudovich},
Nonstationary flow of an ideal incompressible liquid.
\emph{Comput. Math. Math. Phys.} \textbf{3} (1963), 1407--1456.
%Computational Mathematics and Mathematical Physics


\bibitem{Yudovich2}
{\bf V. I. Yudovich},
Uniqueness theorem for the basic nonstationary problem in the dynamics of an ideal incompressible fluid.
\emph{Math. Res. Lett.} \textbf{2} (1995), 27--38.
% Mathematical Research Letters

\end{thebibliography}
\end{document}